\documentclass[12pt, reqno,sumlimits]{article}
\usepackage{amssymb,amscd, epsfig, mathrsfs, xypic, amsmath,amsthm, color} 
\usepackage{blindtext} 

\xyoption{all}
\usepackage{enumerate} 
\numberwithin{equation}{section}

\newcommand{\GGamma}{{G\Gamma}}

\newcommand{\GQ}{{GQ}}
\newcommand{\GP}{{GP}}

\newcommand{\ff}{{f}}

\newcommand{\bbF}{{\mathbb F}}

\newcommand{\bbK}{{\mathbb K}}

\newcommand{\NN}{{\mathbb N}}

\newcommand{\QQ}{{\mathbb Q}}

\newcommand{\ZZ}{{\mathbb Z}}

\newcommand{\sfp }{{\sfp}}

\newcommand{\frakK}{{\mathfrak K}}

\newcommand{\bfz}{{\mathbf z }}

\newcommand{\CH}{{C\!H}}
\newcommand{\CK}{{C\!K}}

\newcommand{\calK}{{\mathcal K}}

\newcommand{\calP}{{\mathcal P}}

\newcommand{\calR}{{\mathcal R}}
\newcommand{\calS}{{\mathcal S}}

\newcommand{\scB}{{\mathscr B}}

\newcommand{\scC}{{\mathscr C}}

\newcommand{\scS}{{\mathscr S}}

\newcommand{\scX}{{\mathscr X}}

\newcommand{\inc}{\hookrightarrow}

\newcommand{\Span}{\operatorname{Span}}

\newcommand{\SG}{SG}

\newcommand{\Fun}{\operatorname{Fun}}

\newcommand{\gr}{\operatorname{gr}}
\newcommand{\fin}{\operatorname{fin}}

\newcommand{\rk}{{\operatorname{rank}}}

\newcommand{\SO}{{{SO}}}
\newcommand{\Sp}{{Sp}}

\newcommand{\Pf}{{\operatorname{Pf}}}
\newcommand{\GX}{{{GX}}}

\newcommand{\IG}{{{IG}}}

\newcommand{\SP}{\calS\calP}

\newcommand{\GT}{G\Theta}

\newcommand{\GB}{{\GT'}}
\newcommand{\GC}{{\GT}}

\newcommand{\OG}{{OG}}

\newcommand{\vece}{\pmb{e}}

\newtheorem{thm}{Theorem}[section]  
 
\newtheorem{lem}[thm]{Lemma}  
\newtheorem{prop}[thm]{Proposition} 
\newtheorem{df-pr}[thm]{Definition-Proposition}

\theoremstyle{definition} 
\newtheorem{defn}[thm]{Definition}
   
\newtheorem{rem}[thm]{Remark}
 
\newtheorem{example}[thm]{Example}

\begin{document}
\title{Double Grothendieck Polynomials for Symplectic and Odd Orthogonal Grassmannians}
\author{Thomas Hudson, Takeshi Ikeda, Tomoo Matsumura, Hiroshi Naruse}
\maketitle
\begin{abstract}
We study the double Grothendieck polynomials of Kirillov--Naruse for the symplectic and odd orthogonal Grassmannians. These functions are explicitly written as sums of Pfaffian and are identified with the stable limits of the fundamental classes of Schubert varieties in the torus equivariant connective $K$-theory of these
isotropic Grassmannians. We also provide a combinatorial description of the ring formally spanned by double Grothendieck polynomials. 
\end{abstract}
\section{Introduction}
In \cite{KrNr}, Kirillov and Naruse  introduced the double Grothendieck polynomials of classical types in order to represent the $K$-theoretic Schubert classes for the corresponding flag varieties. In this paper, we study these functions for the odd orthogonal and symplectic Grassmannians. We first set up a combinatorial model for the ring formally spanned by the stable Schubert classes. This construction is based on previous work for the maximal isotropic Grassmannians (\cite{IkedaNaruse}). By utilizing our result in \cite{HIMN} and equivariant localization techniques (\cite{KostantKumar}, \cite{Krishna}), we will explicitly specify the elements in the ring corresponding to the stable Schubert classes.



Fix integers $k$ and $n$ such that $0\leq k<n$. Let $ \bbF$ be an algebraically closed field of characteristic zero. We equip the vector space $ \bbF^{2n+1}$ with a non-degenerate symmetric bilinear form and $ \bbF^{2n}$ with a symplectic form. We denote respectively by $\OG^k(n)$ and $\SG^k(n)$ the {\it odd-orthogonal} and {\it symplectic} Grassmannians of $(n-k)$-dimensional isotropic subspaces in $ \bbF^{2n+1}$ and $\bbF^{2n}$. We denote both these isotropic Grassmannians by $\IG^k(n)$. Each of such spaces can be realized as a homogeneous space $G/P_k$ where $P_k$ is a parabolic subgroup either of the special orthogonal group $G=\SO_{2n+1}(\bbF)$ or of the symplectic group $G=\Sp_{2n}(\bbF)$. Let $T_n$ be a maximal torus of $G$, $B$ a Borel subgroup of $G$ containing $T_n$, and $B_{-}$ a Borel subgroup opposite to $B$, {\it i.e.}, $B\cap B_{-}=T_n.$ For both $\SO_{2n+1}(\bbF)$ and $\Sp_{2n}( \bbF)$, the Weyl group $W_n$ is isomorphic to the hyperocthedral group given by $\{-1\}^n \rtimes S_n$. Let $W_n^{(k)}$ be the set of minimal length coset representatives of $W_n/W_{n,k}$ where $W_{n,k}$ is the Weyl group  of $P_k$. There is a natural bijection between $W_n^{(k)}$ and the set $\SP^k(n)$ of  $k$-{\it strict partitions} contained in the $(n-k,n+k)$ rectangle (cf. Buch--Kresch--Tamvakis \cite{BuchKreschTamvakis1}). Moveover, in both the odd-orthogonal and symplectic cases, $W_n^{(k)}$ is also in bijection with the set of $T_n$-fixed points in $\IG^k(n)$. For each $\lambda\in \SP^k(n)$, let $e_\lambda$ be the corresponding $T_n$-fixed point and define the Schubert variety $\Omega_\lambda$ associated to $\lambda$ as the closure  of its $B_{-}$-orbit.

We will use the torus-equivariant connective $K$-theory studied by Krishna \cite{Krishna} (cf. Dai--Levine \cite{DaiLevine}). For any non-singular variety $X$ endowed with an action of an algebraic torus $T_n\cong (\bbF^\times)^n$, there is a $\ZZ$-graded algebra $\CK_{T_n}(X)$ over $\CK_{T_n}(pt)$ which can be identified with the {\it graded formal power series ring\/} $\QQ[\beta][[b_1,\ldots,b_n]]_{\gr}$ where $\deg(b_i)=1$ and $\deg(\beta)=-1$. It interpolates the $T_n$-equivariant chow ring $\CH_{T_n}^*(X)$ and the $K$-theory $K_{T_n}(X)$ of $T_n$-equivariant algebraic vector bundles on $X$. By this we mean that it specializes to $\CH_{T_n}^*(X)$ for $\beta=0$ and to $K_{T_n}(X)$ for $\beta=-1$. Since throughout the paper we will work with rational coefficients, here $\CH_{T_n}^*(X)$ stands for the equivariant chow ring with $\QQ$ coefficients and $K_{T_n}(X)$ denotes the equivariant $K$-theory tensored with $\QQ$.

In $\CK_{T_n}(\IG^k(n))$, there exists a {\it fundamental class\/} $[\Omega_\lambda]_{T_n}$
associated to the Schubert variety $\Omega_\lambda$, which specializes to the corresponding $T_n$-equivariant Schubert class  in $\CH^*_{T_n}(\IG^k(n))$ and to the class of the structure sheaf of $\Omega_\lambda$ in $K_{T_n}(\IG^k(n))$. As we are interested in $[\Omega_{\lambda}]_T$, the limit  of the classes $[\Omega_{\lambda}]_{T_n}$ in the graded projective limit of the rings $\CK_{T_n}(\IG^k(n))$, we introduce $\bbK^{\fin}_T(\IG^k)$ a subring which contains all of them (see Definition \ref{dfFiniteCKT}).
More precisely, if we set
\[
\calR_b=\bigcup_{n\geq 0}\QQ[\beta][[b_1,\ldots,b_n]]_{\gr},
\]
then $\bbK^{\fin}_T(\IG^k)$ is the $\calR_b$-module on the formal basis given by $[\Omega_\lambda]_{T}$ indexed by the all $k$-strict partitions $\lambda$.

In order to deal with $\bbK^{\fin}_T(\IG^k)$ algebraically and combinatorially, we introduce the ring $\calK_\infty$, which governs the equivariant $K$-theoretic Schubert calculus of the full-flag varieties of both types B and C. A cornerstone of the construction of $\calK_\infty$ is the ring ${\GGamma} $ introduced in \cite{IkedaNaruse}. It is defined as a subring of $\QQ[\beta][[x]]_{\gr}$ with variables $x=(x_1,x_2,\dots)$ consisting of symmetric graded formal power series with a certain cancellation property (see \S\ref{sec3-1}). Then one sets
\[
\calK_\infty=\GGamma \otimes_{\QQ[\beta]} \calR_a\otimes_{\QQ[\beta]} \calR_b.
\]
This construction is analogous to the one of the ring $\calR_\infty$ introduced as an algebraic model for the torus equivariant cohomology of full-flag varieties of type B, C, and D in \cite{IkedaMihalceaNaruse}. There are two (called {\it left\/} and {\it right\/}) commuting actions of $W_\infty= \bigcup_{n\geq 1}W_{n}$ on $\calK_\infty$. The left action commutes with $\calR_a$ while the right one commutes with $\calR_b$. The ring $\calK^{(k)}_\infty$ is then the invariant subring with respect to the right action of the group  $W_{(k)}=\bigcup_{n> k}W_{n,k}$. In particular, the invariant ring $\calK^{(0)}_\infty$ coincides with $\GGamma \otimes_{\QQ[\beta]}\calR_b$.

The first main result is the following algebraic realization of $\bbK^{\fin}_T(\IG^k)$, which we obtain by comparing the GKM descriptions of $\calK^{(k)}_\infty$ and $\CK_{T_n}(\IG^k(n))$.

\vspace{3mm}
\noindent{\bf Theorem} (Theorem \ref{thm:iso} below)
{\it 
There are isomorphisms of graded $\calR_b$-algebras
\[
\calK^{(k)}_\infty \cong \bbK^{\fin}_T(\OG^k) \ \ \mbox{ and } \ \ \ 
\calK^{(k)}_\infty \cong \bbK^{\fin}_T(\SG^k).
\]
}

\vspace{-5mm}
The second main result is the description of the those elements in  $\calK_\infty^{(k)}$ corresponding to $[\Omega_\lambda]_T$ under the above isomorphisms. For each $k$-strict partition $\lambda$, we define the functions ${}_k\GB_{\!\lambda}$ and ${}_k\GC_{\lambda}$ in $\calK_\infty^{(k)}$ via a Pfaffian formula (Definition \ref{dfGXPf}) analogous to the one obtained in \cite{HIMN} for the Schubert classes for $\CK_{T_n}(\IG^k(n))$ (cf.\! Theorem \ref{thmPf}). This allows us to show that those functions correspond to the limits of the Schubert classes (Theorem \ref{thm:iso}). Note that the Pfaffian formula in \cite{HIMN} describes the Schubert classes in terms of the {\it (equivariantly shifted) special classes} which we also identify with certain functions in $\calK_{\infty}^{(k)}$ (Lemma \ref{propGT=Seg}). It is worth mentioning that when $\beta=0$, the functions ${}_k\GC_{\lambda}$ specialize to Wilson's double $\vartheta$-functions \cite{WilsonThesis} which coincide with the type C double Schubert polynomials (see \cite{IkedaMatsumura}).

For the maximal isotropic Grassmannians in type B and C, the isomorphisms in the above theorem were obtained in \cite{IkedaNaruse}. There have also been introduced the so-called $\GP$- and $\GQ$-functions in a combinatorially explicit formula, which represent the $K$-theory Schubert classes. Thus by our result we find that, in the case $k=0$,  they coincide with ${}_0\GB_{\!\lambda}$ and ${}_0\GC_{\lambda}$ respectively, hence  establishing their Pfaffian formula simultaneaously. 


Kirillov--Naruse \cite{KrNr} constructed the double Grothendieck polynomials representing the equivariant $K$-theoretic Schubert classes for the full-flag varieties of type B and C. One can prove that our new functions consist of an appropriate subfamily of their polynomials, simultaneously establishing their explicit closed formula in the form of Pfaffians. We also remark that by their construction, those functions are polynomials in $x$, $a$, $b$ variables and $\beta$ with coefficients in $\ZZ_{\geq 0}$, once we choose a positive integer $n$ and set $x_i=0$ for all $i > n$. 

Given the previous remark related to Kirillov--Naruse's constructions, let us mention some combinatorial problems that still remain with regard to our $\GC$- and $\GB$-functions.
When $k=0$, it was proved in \cite{IkedaNaruse} that $\GQ$- and $\GP$-functions are given in terms of {\it shifted set-valued tableaux}, generalizing the classical tableaux formula of Schur $Q$- and $P$-functions (cf. \cite[\S III.8]{MacdonaldHall}). When $\beta=0$ and $b_i=0$ for all $i$ but for general $k$, Tamvakis \cite{Tamvakis2011Crelle} obtained an analogous combinatorial formula for the theta polynomials. Thus it would be an interesting problem to find similar combinatorial formulas for our $\GC$- and $\GB$-functions. 
It is also worth mentioning that, again from Kirillov--Naruse's construction, we can deduce the fact that $\GC$-functions ({\it resp.} $\GB$-functions) can be expanded into a linear combination of $\GQ$-functions ({\it resp.} $\GP$-functions) with coefficients in $\ZZ[\beta][a,b]$. It is an open problem to find an explicit formula for the coefficients of the expansions. Note that when $\beta=0$, such formula has been known in the work of Buch--Kresch--Tamvakis \cite[Theorem 4]{BuchKreschTamvakis1} and Tamvakis--Wilson \cite[Corollary 2]{TamvakisWilson}.

Finally, we note that recently Anderson \cite{Anderson2016} proved a formula describing a larger family of $K$-theoretic Schubert classes including the case of type D isotropic Grassmannians, generalizing our results for type B and C in \cite{HIMN}. Given this, we expect that it is possible to define the functions analogous to $\GC$ and $\GB$ explicitly, extending the main results of this paper to type D.

This paper is organized as follows. In Section \ref{sec2}, we recall the combinatorics related to the Weyl group of type B and C. In Section \ref{sec3}, we define the rings $\GGamma$, $\calK_{\infty}$ and $\calK_{\infty}^{(k)}$ and introduce the functions ${}_k\GB_{\!\lambda}$ and ${}_k\GC_{\lambda}$. In Section \ref{sec4}, we recall the definitions of the odd orthogonal and symplectic Grassmannians together with the torus equivariant $K$-theory Schubert classes. In Section \ref{sec5}, we compare the GKM descriptions of $\calK_{\infty}^{(k)}$ and $\CK_{T_n}(\OG^k(n))$ and define a map relating those two algebras. Finally we establish our main results, namely the isomorphisms in the above theorem, and the fact that ${}_k\GB_{\!\lambda}$ and ${}_k\GC_{\lambda}$ represent the Schubert classes.

\section{Preliminary on combinatorics of type B and C}\label{sec2}
We fix notation on Weyl groups, root systems, and $k$-strict partitions.
\subsection{Weyl group and root systems}\label{sec2-1}
Let $W_{\infty}$ be the infinite hyperoctahedral group, which is defined by the generators  $s_i, i=0,1,\ldots$, and the relations 
\begin{equation}\label{coxeter relation}
\begin{array}{c}
s_i^2=e \ \ \ \ (i\geq 0) ,\ \ \ \ \ \ \ s_is_j=s_js_i\ \ \ \ (|i-j|\geq 2),\\
s_0s_1s_0s_1=s_1s_0s_1s_0, \ \  s_is_{i+1}s_i=s_{i+1}s_is_{i+1}\ \ \ \ (i\geq 1).\\
\end{array}
\end{equation}
We identify $W_{\infty}$ with the group of {\it signed permutations}, {\it i.e.} all permutations $w$ of $\ZZ \backslash \{0\}$ such that $w(i)\not= i$ for only finitely many $i\in \ZZ \backslash \{0\}$, and $\overline{w(i)}=w(\bar{i})$ for all $i$. In this context $\bar{i}$ stands for $-i$. The generators, often referred to as \emph{simple reflections}, are identified with the transpositions $s_{0}=(1,\bar{1})$ and $s_{i}=(i+1,i)(\overline{i},\overline{i+1})$ for $i\geq 1$. The \emph{one-line notation} of an element $w\in W_\infty$ is the sequence $w=(w(1),w(2),w(3),\cdots)$.  The \textit{length} of $w\in W_\infty$ is denoted by $\ell(w)$.

For each nonnegative integer $k$, let $W_{(k)}$ be the subgroup of $W_{\infty}$ generated by all $s_i$ with $i\not=k$. Let $W_\infty^{(k)}$ be the set of minimum length coset representatives for $W_{\infty}/W_{(k)}$, which is described as
\[
W_\infty^{(k)} = \{ w \in W_\infty \ |\ \ell(ws_i)>\ell(w) \mbox{ for all } i\not=k\}.
\]
An element of $W_\infty^{(k)}$ is called \emph{$k$-Grassmannian} and it is given by  the following one-line notation:  
\begin{equation}\label{oneline for k grass}
\begin{array}{c}
w=(v_1,\cdots, v_k, \overline{\zeta_1},\cdots, \overline{\zeta_s}, u_1,u_2,\cdots);\\
0<v_1<\cdots <v_k, \ \ \ \overline{\zeta_1}<\cdots< \overline{\zeta_s} <0<u_1<u_2<\cdots,
\end{array}
\end{equation}
where $s\geq 0$. For example, $(1,3,\bar4, 2,5,6,7,\cdots)$ is a $2$-Grassmannian element in $W_{\infty}$. 

Upon a choice of an integer $n\geq 1$, we let $W_n$ be the subgroup of $W_{\infty}$ generated by $s_0, s_1,\dots, s_{n-1}$. Or, equivalently, it consists of the elements $w \in W_{\infty}$ such that $w(i)=i$ for all $i>n$. We write the one-line notation of $w \in W_n$ as the finite sequence $(w(1),w(2),\cdots, w(n))$. We set $W_{n,(k)}:=W_n \cap W_{(k)}  $ and $W_n^{(k)}:=W_n \cap W_\infty^{(k)}$ so that $W_n^{(k)}\cong W_n/W_{n,(k)}$.

Let $L:=\bigoplus_{i=1}^\infty \ZZ \varepsilon_i$ be the free module generated by $\varepsilon_i, i\in \ZZ_{>0}$. We define the set $\Delta^+$ of {\em positive roots} in $L$ as 
\begin{eqnarray*}
&&\Delta^+ := \{ \varepsilon_i, 1\leq i \} \cup \{\varepsilon_j \pm \varepsilon_i\ |\ 1\leq i <  j\} \ \ \ \mbox{for type B and}\\
&&\Delta^+ := \{ 2\varepsilon_i, 1\leq i \} \cup \{\varepsilon_j \pm \varepsilon_i\ |\ 1\leq i <  j\} \ \ \ \mbox{for type C.}
\end{eqnarray*}
Let $s_\alpha\in W_\infty$ denote the simple reflection associated with the positive root $\alpha\in \Delta^+$. Let $\Delta^+_n:= \Delta^+\cap \ \Span_{\ZZ}\{\varepsilon_1,\dots, \varepsilon_n\}$, then for $\alpha \in \Delta^+_n$, the simple reflection $s_{\alpha}$ is an element of $W_n$.
\subsection{$k$-strict partitions}\label{sec2-2}
Fix a nonnegative integer $k$. A {\it $k$-strict partition} is an infinite sequence $(\lambda_1,\lambda_2,\cdots)$ of non-increasing nonnegative integers such that all but finitely many $\lambda_i$'s are zero, and such that $\lambda_i>k$ implies $\lambda_i>\lambda_{i+1}$. The {\it length} of a $k$-strict partition is the largest index $r$ such that $\lambda_r\not=0$. We often denote a $k$-strict partition $\lambda$ with length at most $r$ by $\lambda=(\lambda_1,\dots, \lambda_r)$. We use the following notation for the sets of $k$-strict partitions:
\begin{eqnarray*}
\SP^{k}&:& \mbox{the set of all $k$-strict partitions; }\\
\SP_r^{k}&:&   \mbox{the set of all $k$-strict partitions with length at most $r$;}\\
\SP^{k}(n)&:&   \mbox{the set of all $k$-strict partitions with length at most $n-k$ and $\lambda_1\leq n+k$}.
\end{eqnarray*}
 There is a bijection from $W_{\infty}^{(k)}$ to $\SP^{k}$ given as follows. Let $w \in W_\infty^{(k)}$ be an element with the one-line notation (\ref{oneline for k grass}). Let $\nu=(\nu_1,\nu_2,\dots)$ be a partition given by $\nu_i=\sharp\{p \ |\ v_p>u_i\}$. We define a $k$-strict partition $\lambda$ by setting $\lambda_i=\zeta_i+k$ if $1\leq i\leq s$ and  $\lambda_i=\nu_{i-s}$ if  $s+1\leq i$. This bijection can be restricted to $W_n^{(k)}$, {\it i.e.} $W_n^{(k)} \cong \SP^k(n)$. See Buch--Kresch--Tamvakis \cite{BuchKreschTamvakis1} for details. Through this bijection $W_\infty^{(k)}\cong \SP^k$, the Weyl group $W_\infty$ acts naturally on $\SP^k$. Similarly, $W_n$ acts on $\SP^k(n)$ via the bijection $W_n^{(k)}\cong \SP^k(n)$.

\section{Rings of double Grothendieck polynomials
}\label{sec3}
We first define the ring $\GGamma$ and introduce its formal basis $\{\GP_\lambda(x)\}_{\lambda\in \SP^0}$. The ring $\calK_\infty$ of double Grothendieck polynomials is then introduced. On this ring we have a natural (right) action of $W_\infty$. For each nonnegative integer $k$, we define the functions ${}_k\GC_\lambda$ and ${}_k\GB_{\!\!\lambda}$ associated to all $\lambda\in \SP^k$ as elements in the invariant subring $\calK^{(k)}_\infty$ with respect to the action of the subgroup $W_\infty^{(k)}$ of $W_\infty$.
\subsection{The ring $\GGamma$ and its formal basis}\label{sec3-1}
Let $x=(x_1,x_2,\ldots)$ be a sequence of  indeterminates of degree $1$. Let $\QQ[\beta]$ be the polynomial ring in a variable $\beta$ with $\deg \beta =-1$. Denote by $\QQ[\beta][[x]]_m$  the set of all power series in $x$ with coefficients in $\QQ[\beta]$ of degree $m  \in \ZZ$. Then $\QQ[\beta][[x]]_{\gr}:=\bigoplus_{m\in \ZZ}\QQ[\beta][[x]]_{m}$ is a graded $\QQ[\beta]$-algebra, called \emph{the ring of graded formal power series}. We denote the ring of graded formal power series in a finite sequence $(x_1,\dots, x_n)$ of variables by $\QQ[\beta][[x_1,\dots, x_n]]_{\gr}$. Throughout the paper, we denote, for any variables $u$ and $v$, 
\[
u\oplus v:=u+v+\beta uv, \ \ \ u\ominus v := \frac{u-v}{1+\beta v}, \ \ \ \bar u:=\frac{-u}{1+\beta u}.
\]
where $1+\beta u = \sum_{m\geq 0} \beta^m u^m$.
\begin{defn}[\cite{IkedaNaruse}]\label{GGcond}
We denote by $\GGamma_n$ the graded subring of $\QQ[\beta][[x_1,\ldots,x_n]]_{\gr}$ whose elements are the series $f(x)$ such that:
\begin{itemize}
\item[(1)] $f(x)$ is symmetric in $x_1,\ldots,x_n$;
\item[(2)] $f(t,\bar t,x_3,x_4,\dots,x_n)=f(0,0,x_3,x_4,\dots,x_n)$.
\end{itemize}
\end{defn}
These rings form a projective system with respect to the degree preserving homomorphism $\GGamma_{n+1}\to \GGamma_n$ given by $x_{n+1}=0$. Let us denote by $\GGamma$ its graded projective limit. We can identify $\GGamma$ with the subring of $\QQ[\beta][[x]]_{\gr}$ defined by the conditions (1)' $f(x)$ is symmetric in $x_i, i\in \ZZ_{>0}$ and (2)' $f(t,\bar t,x_3,x_4,\dots)=f(0,0,x_3,x_4,\dots)$, analogous to (1) and (2) above. Note that our $\GGamma$ is a completion of the one defined in \cite{IkedaNaruse}. 
\begin{defn}[\cite{IkedaNaruse}]
For each strict partition $\lambda=(\lambda_1,\dots,\lambda_r)$ of length $r$ in $\SP_n^0$, the corresponding $GP$-functions are defined by  
\[
\GP_{\lambda}(x_1,\dots,x_n) = \frac{1}{(n-r)!} \sum_{w\in S_n} w\left[ x_1^{\lambda_1}\cdots x_r^{\lambda_r}\prod_{i=1}^r\prod_{j=i+1}^n\frac{x_i\oplus x_j}{x_i\ominus x_j} \right].
\]

\end{defn}
The polynomial $\GP_{\lambda}(x_1,\dots,x_n)$ is an element of $\GGamma_n$ for all $n$, and these define an element $\GP_{\lambda}(x)$ in $\GGamma$ as the limit.

The next proposition will be used in the proof of Lemma \ref{lemfullinj} at $\S\ref{sec5}$.
\begin{prop}\label{propfb}
Any homogeneous element $f(x)$ of $\GGamma$ with degree $m$ is uniquely expressed as a possibly infinite linear combination
\[
f(x) = \sum_{\lambda \in \SP^0} c_{\lambda} \GP_{\lambda}(x), \ \ \ c_{\lambda} \in \QQ[\beta]_{m-|\lambda|}.
\]
\end{prop}
\begin{proof}
The claim follows formally from the fact that $\GP_{\lambda}(x_1,\dots, x_n) \ (\lambda\in \SP_n^0)$ form a formal basis of $\GGamma_n$. That is, a homogeneous element $f(x_1,\dots,x_n)$ in $\GGamma_n$ of degree $r$ is uniquely expressed as a possibly infinite linear combination
\[
f(x_1,\dots, x_n) = \sum_{\lambda \in \SP_n^0} c_{\lambda} \GP_{\lambda}(x_1,\dots,x_n), \ \ \ c_{\lambda} \in \QQ[\beta]_{r-|\lambda|}.
\]
This fact is a slight modification of Theorem 3.1 \cite{IkedaNaruse}: here we must work in the ring of graded formal power series. We leave the details to the reader since it is parallel to the original one.
\end{proof}
\subsection{Rings of double Grothendieck polynomials}\label{sec3-2}
For infinite sequences of variables $a=(a_1,a_2,\dots)$ and $b=(b_1,b_2,\dots)$, consider the rings
\begin{eqnarray*}
\calR_a&:=& \bigcup_{m=0}^{\infty} \QQ[\beta][[a_1,\dots, a_m]]_{\gr}, \ \ \ \calR_b:= \bigcup_{m=0}^{\infty} \QQ[\beta][[b_1,\dots,b_m]]_{\gr}.
\end{eqnarray*}
and define the $\calR_b$-algebra $\calK_{\infty}$ by
\[
\calK_{\infty}:=\GGamma\otimes_{\QQ[\beta]}\calR_a \otimes_{\QQ[\beta]} \calR_b.
\]

We define an action of $W_{\infty}$ on $\calK_\infty$ as follows. 
For $f(x;a)\in \GGamma\otimes_{\QQ[\beta]}\calR_a$ we set
\begin{eqnarray*}
(s_0 f)(x;a)&=&f(a_1,x_1,x_2,\ldots; \overline{a_1}, a_2, \ldots)\\
(s_i f)(x;a)&=&f(x_1,x_2,\ldots; a_1, a_2, \ldots,a_{i+1},a_i,\ldots)\quad (i\geq 1)
\end{eqnarray*}
and extend these as automorphisms of $\calR_b$-algebras. One can check that this gives an action of $W_\infty$ on $\calK_\infty$. For example, $s_0^2=1$ follows from the definition of $\GGamma$. We call this the {\it right} action, while the {\it left} action is similarly defined by replacing the roles of $a$ and $b$. In this paper, we only use the right one.


Now for each $k\geq 0$, we define $\calK_{\infty}^{(k)}$ to be the subalgebra of $\calK_{\infty}$ invariant under the $W_{(k)}$-action:
\[
\calK_{\infty}^{(k)} := \calK_{\infty}^{W_{(k)}} = (\GGamma\otimes_{\QQ[\beta]}\calR_a)^{W_{(k)}}  \otimes_{\QQ[\beta]} \calR_b.
\] 
Note that the subring $(\GGamma\otimes_{\QQ[\beta]}\calR_a)^{W_{(k)}}$ of $\GGamma\otimes_{\QQ[\beta]}\calR_a$ invariant under the action of $W_{(k)}$, is contained in $\GGamma[[a_1,\dots,a_k]]_{\gr}$ since each element of $\calR_a$ involves only finitely many $a_i$'s.
\begin{rem}
The ring $\calK_{\infty}$ (resp. $\calK_{\infty}^{(k)}$) is the $K$-theoretic version of $\calR_{\infty}$ introduced in \cite{IkedaMihalceaNaruse} (resp. $\calR_\infty^{(k)}$ in \cite{IkedaMatsumura}). The corresponding \emph{double Grothendieck polynomials} constructed by Kirillov-Naruse \cite{KrNr} represent the $K$-theoretic equivariant Schubert classes. 
\end{rem}
\subsection{Definitions of $\GC_{\lambda}$ and $\GB_{\lambda}$}\label{sec3-3}
\begin{defn}\label{GTsingle}
We introduce the functions ${}_k\GC_m(x,a)\in \GGamma[[a_1,\dots,a_k]]_{\gr}$ $(m\in \ZZ)$ by the following generating function
\begin{eqnarray*}
{}_k\GC(x,a;u):=\sum_{m\in \ZZ} {}_k\GC_m(x,a) u^m =\frac{1}{1+\beta u^{-1}} \prod_{i=1}^{\infty} \frac{1 + (u+\beta) x_i}{1 + (u+\beta)\bar x_i}\prod_{i=1}^{k} (1 + (u+\beta) a_i).
\end{eqnarray*}
\end{defn}
\begin{lem}\label{lemGTinv} 
The functions ${}_k\GC_m(x,a)$ are elements of $(\GGamma\otimes_{\QQ[\beta]}\calR_a)^{W_{(k)}}$ for all $m\in \ZZ$.
\end{lem}
\begin{proof}
First, we observe that ${}_k\GC(x,a;u)$ satisfies the conditions (1)' and (2)'. Thus ${}_k\GC_m(x,a)$ are in $\GGamma\otimes_{\QQ[\beta]}\calR_a$. Next we see that ${}_k\GC(x,a;u)$ is symmetric in $a_1,\dots, a_k$ and hence ${}_k\GC_m(x,a)$ are  invariant under the actions of $s_1^{a}, \dots,s_{k-1}^{a}$. Furthermore, it is obvious from definition that $s_0^{a}({}_k\GC(x,a;u)) = {}_k\GC(x,a;u)$. Thus ${}_k\GC_m(x,a)$ are also invariant under the action of  $s_0^a$ (cf.  \cite[Proposition 5.1]{IkedaMatsumura}).
\end{proof}
Lemma \ref{lemGTinv} allows us to define functions ${}_k\GC_m^{(\ell)}(x,a|b)$ and ${}_k\GB_m^{(\ell)}(x,a|b)$ in $\calK_{\infty}^{(k)}$ as follows.
\begin{defn}\label{df:GCGBdouble}
We define the elements ${}_k\GC_m^{(\ell)}(x,a|b)$ and ${}_k\GB_m^{(\ell)}(x,a|b)$ in $\calK_{\infty}^{(k)}$ $(m,\ell\in \ZZ)$ by the following generating functions
\begin{eqnarray*}
\sum_{m\in \ZZ} {}_k\GC_m^{(\ell)}(x,a|b) u^m  
&=&\begin{cases}
{}_k\GC(x,a;u) \displaystyle\prod_{i=1}^{|\ell|}\dfrac{1}{1 + (u+\beta)\bar b_i} & (\ell <0),\\
{}_k\GC(x,a;u) \displaystyle\prod_{i=1}^{\ell}(1 + (u+\beta) b_i)& (\ell \geq 0),
\end{cases}\\
\sum_{m\in \ZZ} {}_k\GB_m^{(\ell)}(x,a|b) u^m 
&=& \begin{cases}
{}_k\GC(x,a;u) \displaystyle\prod_{i=1}^{|\ell|}\dfrac{1}{1 + (u+\beta)\bar b_i} & (\ell <0),\\
\displaystyle\frac{1}{2+\beta u^{-1}}{}_k\GC(x,a;u) \displaystyle\prod_{i=1}^{\ell}(1 + (u+\beta) b_i)& (\ell \geq 0).
\end{cases}
\end{eqnarray*}
We denote ${}_k\GX_m^{(\ell)}={}_k\GC_m^{(\ell)}(x,a|b)$ for type C and ${}_k\GX_m^{(\ell)}={}_k\GB_m^{(\ell)}(x,a|b)$ for type $B$. 
\end{defn}
\begin{rem}\label{rem:GTminus}
A direct computation shows that ${}_k\GC_m^{(\ell)}(x,a|b) = (-\beta)^{-m}$ for each $m\leq 0$. Moreover, if $\ell\geq 0$, we have
\[
{}_k\GB_m^{(\ell)}(x,a|b)= \frac{1}{2} \sum_{s\geq 0} \left(\displaystyle\frac{-\beta}{2}\right)^s{}_k\GC_{m+s}^{(\ell)}(x,a|b).
\]
\end{rem}
In order to define $\GC_{\lambda}$ and $\GB_{\lambda}$ in terms of Pfaffians, we prepare some notations related to $k$-strict partitions.
\begin{defn}
For $\lambda\in \SP^{k}$, we set
\begin{eqnarray*}
C(\lambda)&:=&\Big\{(i,j) \in \NN^2 \ |\ 1\leq  i < j, \ \ \lambda_i+\lambda_j>2k+j-i\Big\}, \\
\gamma_j&:=& \sharp\Big\{ i \in \NN \ |\ (i,j)\in C(\lambda)\Big\}\ \ \ \ \mbox{for each $j>0$}.
\end{eqnarray*}
and define the associated \textit{characteristic index} $\chi=(\chi_1,\chi_2,\dots)$ by 
\[
\chi_j:=\lambda_j-j+\gamma_j-k.
\]
\end{defn}

\begin{rem}
For $k=0$, a $k$-strict partition $\lambda$ is called a strict partition. In this case, one has $\chi_i=\lambda_i-1$ for all $i=1,\dots, r$,  where $r$ is the length of $\lambda$. 
\end{rem}
\begin{rem}\label{chilambda}
We have $\chi_i+\chi_j\geq0$ if and only if $\lambda_i+\lambda_j>2k+j-i$ for $i<j$ (see \cite[Lemma 3.3]{IkedaMatsumura}). As a consequence one has $C(\lambda)=\{(i,j) \ |\ 1\leq  i < j, \ \ \chi_i + \chi_j \geq 0\}$.
\end{rem}
\begin{defn}\label{dfPfSymb}
Let $\lambda$ be a $k$-strict partition of length $r$. We set 
\begin{eqnarray*}
\Delta_r &:=& \{(i,j) \ |\ 1\leq i < j\leq r\}\\
D(\lambda)  &:=& \Big\{ (i,j) \in \Delta_r\ |\  \chi_i + \chi_j < 0\Big\} = \Delta_r \backslash C(\lambda).
\end{eqnarray*}
For each $I \subset D(\lambda)$, we set
\[
a_i^I:=\sharp\{j \ |\ (i,j)\in I\},  \ \ \   c_j^I:=\sharp\{i \ |\ (i,j)\in I\}, \ \ \ \mbox{and}\ \ \ d_i^I:=a_i^I-c_i^I.
\]
Denote $m:=r$ if $r$ is even and $m:=r+1$ if $r$ is odd. Consider the following rational function of variables $t_i$ and $t_j$ ($1\leq i,j\leq m$): recall $\bar t=\frac{-t}{1+\beta t}$ and set
\[
F_{i,j}^I(t):=\frac{1}{(1 + \beta  t_i)^{m-i-c_i^I-1}} \frac{1}{(1 + \beta  t_j)^{m-j-c_j^I}} \frac{1 - \bar  t_i/\bar t_j}{1 -  t_i/\bar  t_j}.
\]
Let $\ff_{pq}^{ij,I}$ be the coefficient of the expansion of $F_{i,j}^I(t)$ as the following Laurent series
\begin{equation}\label{defLaurent}
F_{i,j}^I(t)= \sum_{p\geq 0, \atop{p+q\geq 0}} \ff_{pq}^{ij,I} t_i^pt_j^q.
\end{equation}
\end{defn}

\begin{defn}\label{dfGXPf}
Let $\lambda \in \SP^k(n)$ of length $r$ with its characteristic index $\chi$. Let $m:=r$ if $r$ is even and $m:=r+1$ if $r$ is odd. We define the element ${}_k\GX_{\lambda}\in \calK_{\infty}^{(k)}$ for type C and B by
\[
{}_k\GX_{\lambda}= \sum_{I \subset D(\lambda)} \Pf\left(\sum_{p,q\in \ZZ \atop{p\geq 0, p+q\geq 0}} \ff_{pq}^{ij,I} {}_k\GX_{\lambda_i+d_i^I+p}^{(\chi_i)}\cdot {}_k\GX_{\lambda_j+d_j^I+q}^{(\chi_j)}\right)_{1\leq i<j\leq m }.
\]
We denote ${}_k\GC_{\lambda}(x,a|b):={}_k\GX_{\lambda}(x,a|b)$ for type C and ${}_k\GB_{\lambda}(x,a|b):={}_k\GX_{\lambda}(x,a|b)$ for type B. If there is no fear of confusion, we denote ${}_k\GX_{\lambda}$ simply by $\GX_{\lambda}$. 
\end{defn}
\begin{example}
For a $k$-strict partition $\lambda=(\lambda_1)$ of length $1$, the corresponding characteristic index is $\chi=(\lambda_1-k-1)$. In this case, we have ${}_k\GC_{(\lambda_1)}(x,a|b)={}_k\GC_{\lambda_1}^{(\lambda_1-k-1)}(x,a|b)$ and ${}_k\GB_{(\lambda_1)}(x,a|b) = {}_k\GB_{\lambda_1}^{\,(\lambda_1-k-1)}(x,a|b)$.
\end{example}

\section{$K$-theoretic Schubert classes of $\IG^k(n)$
}\label{sec4}

\subsection{Preliminary on equivariant connective $K$-theory}\label{sec4-1}
First of all, let us briefly recall the some facts about torus equivariant connective $K$-theory and the $K$-theoretic Segre classes which will be needed later. 
 For more details readers are referred to \cite{Krishna}, \cite[Section 2]{HIMN} and, for an alternative presentation which has no restriction on the characteristic of the base field, to \cite[Appendix]{Anderson2016}.

Let $T_n$ be a standard algebraic torus $(\bbF^{\times})^n$. For a smooth variety $X$ with an action of $T_n$, its equivariant connective $K$-theory is denoted by $\CK^*_{T_n}(X)$. In this paper, we work with the rational coefficients $\QQ$ instead of the integral coefficients $\ZZ$. Let $E$ be a $T_n$-equivariant vector bundles, then its $i$-th $T_n$-equivariant Chern class is denoted by $c_i(E)$, while its total Chern class of $E$ is denoted by $c(E;u)=\sum_{i=0}^{\rk E} c_i(E)u^i$. If $F$ is another $T_n$-equivariant vector bundle, then the total Chern class of the virtual bundle $E-F$ is defined by $c(E-F;u)=\sum_{i\geq 0} c_i(E-F)u^i=c(E;u)/c(F;u)$. 

The equivariant connective $K$-theory $\CK^*_{T_n}(X)$ is in fact a graded algebra over $\CK^*_{T_n}(pt)$ which we identify with a ring of graded formal power series as follows. We regard $\varepsilon_1,\dots, \varepsilon_n$ as the standard basis of the character group of $T_n$. Let $L_i$ be the one dimensional representation of $T_n$ with character $\varepsilon_i$. We use the following isomorphism
\begin{equation}\label{eq:CK(pt)}
\CK^*_{T_n}(pt) \to \QQ[\beta][[b_1,\dots,b_n]]_{\gr}; \ \ c_1(L_{i}) \mapsto b_i
\end{equation}
of graded $\QQ[\beta]$-algebras (\cite[\S 2.6]{Krishna}). For simplicity, we denote $\CK^*_{T_n}:=\CK^*_{T_n}(pt)$. We regard a $\CK^*_{T_n}$-algebra as a $\QQ[\beta][[b]]_{\gr}$-algebra or $\calR_b$-algebra via the projection $\QQ[\beta][[b]]_{\gr} \to \CK^*_{T_n}$ or $\calR_b \to \CK^*_{T_n}$ defined by $b_i=0$ for all $i>n$ respectively.
\begin{rem}
The ring $\CK^*_{T_n}=\QQ[\beta][[b_1,\dots,b_n]]_{\gr}$ specializes to $K_{T_n}({pt})=\QQ[[b_1,\dots, b_n]]$ at $\beta=-1$. It can be also identified with a completion of the representation ring $R(T_n)$ of $T_n$ with rational coefficients as follows. First recall that we have
\[
R(T_n)\otimes_{\ZZ}\QQ=\QQ[e^{\pm \varepsilon_1},\dots, e^{\pm \varepsilon_n}],
\]
where the class of $L_i$ corresponds to $e^{-\varepsilon_i}$. Now $1 - e^{\varepsilon_i}$ is identified with the first Chern class $b_i=c_1(L_i)$ of $L_i$ (\textit{cf.} Krishna \cite[Theorem 7.3]{Krishna0}).
\end{rem}
Although the Segre classes can be defined geometrically, we take the following definition in terms of a generating function due to \cite{HIMN}.
\begin{defn}\label{dfrelsegre}
For a virtual bundle $E-F$, we define the {\it $T_n$-equivariant relative Segre classes} $\scS_{m}(E-F)$  $(m\in \ZZ)$  in $\CK^*_{T_n}(X)$ by 
\begin{equation}\label{segre vir}
\scS(E-F;u):=\sum_{m\in \ZZ} \scS_{m}(E-F) u^{m}= \frac{1}{1 + \beta u^{-1}} \frac{c(E - F;\beta)}{c(E-F;-u)}.
\end{equation}
\end{defn}
\begin{rem}\label{remgrlw}
The Chern classes of the connective $K$-theory are governed essentially by the {\it multiplicative formal group law} $x\oplus y = x+y+\beta xy$. Namely if $L$ and $M$ are line bundles, then $c_1(L\otimes M) = c_1(L) \oplus c_1(M)$ and $c_1(L^{\vee}) = \frac{-c_1(L)}{1+\beta c_1(L)}$. This being said, the identity $\frac{1+\beta x}{1- xu} = \frac{1}{1+(u+\beta)x}$ implies 
\[
\scS(E-F;u)= \frac{1}{1 + \beta u^{-1}} \frac{c(F^{\vee};u+\beta)}{c(E^{\vee};u+\beta)}.
\]
\end{rem}
\subsection{Schubert varieties and the stability of their classes}\label{sec4-2}
For an integer $n \geq k$, let $E^{(n)}$ be a vector space $\bbF^{2n}$ or $\bbF^{2n+1}$ of dimension $2n$ or $2n+1$ respectively. We fix bases by
\[
\bbF^{2n}=\Span\{\vece_{\bar i}, \vece_i\:|\;1\leq i\leq n\}\ \ \ \ \mbox{ and } \ \ \ 
\bbF^{2n+1}=\Span\{\vece_{\bar i}, \vece_i\:|\;1\leq i\leq n\} \oplus \Span\{\vece_0\},
\]
together with  the symplectic form $\sum_{i=1}^n \vece_i^*\wedge \vece_{\bar i}^*$ and the non-degenerate symmetric form $\vece_0^*\otimes \vece_0^*+\sum_{i=1}^n \vece_i^*\otimes \vece_{\bar i}^*$ respectively where $\{\vece_i^*\}$ denotes the dual basis of $\{\vece_i\}$. We define the action of $T_n$ on $E^{(n)}$ by setting the weights of $\vece_i$ and $\vece_{\bar i}$ to be $\varepsilon_i$ and $-\varepsilon_i$ respectively for $1\leq i\leq n$, while the weight of $\vece_0$ is $0$. This identifies $T_n$ with maximal tori of $\Sp_{2n}(\bbF)$ and $\SO_{2n+1}(\bbF)$ respectively.

Let $\IG^k(n)$ be the Grassmannians of $n-k$ dimensional isotropic subspaces in $E^{(n)}$, {\it i.e.} $\IG^k(n)$ is the symplectic Grassmannian $\SG^k(n)$ if $E^{(n)}=\bbF^{2n}$ (type C) and the odd orthogonal Grassmannian $\OG^k(n)$ if $E^{(n)}=\bbF^{2n+1}$ (type B). Consider the subspaces of $E^{(n)}$ 
\begin{eqnarray*}
&&F^{\ell} = \Span\{\vece_n,\dots, \vece_{\ell+1}\} \  \ \ (0\leq \ell \leq n),\\
&&F^{-\ell} = (F^0)^{\perp} \oplus \Span\{\vece_{\bar1},\cdots, \vece_{\bar\ell}\} \ \ \ (1\leq \ell \leq n).
\end{eqnarray*}
It is known that the Schubert varieties in $\IG^k(n)$ are described as follows. For a $k$-strict partition $\lambda\in \SP^k(n)$ with length $r$ and characteristic index $\chi$, the associated Schubert variety $\Omega_{\lambda}^X$ in $\IG^k(n)$ is given by 
\[
\Omega_{\lambda}^X = \{ U\in \IG^k(n) \ |\ \dim (U \cap F^{\chi_i}) \geq i, \ \ \  i=1,\dots, r\},
\]
where we write $X=C$ for the symplectic case, and $X=B$ for the odd orthogonal case. 


Since the Schubert variety $\Omega_{\lambda}^X$ is $T_n$-stable, it defines the $T_n$-equivariant class $[\Omega_\lambda^X]_{T_n}$ in $\CK^*_{T_n}(\IG^k(n))$. As a $\CK^*_{T_n}$-module, $\CK^*_{T_n}(\IG^k(n))$ is freely generated by $[\Omega_\lambda^X]_{T_n}, \lambda\in\SP^k(n)$. See \cite{Krishna}. 

Let $E^{(n)} \to E^{(n+1)}$ be the injective linear map defined by the inclusion of the basis elements. It induces an embedding $j_n: \IG^k(n) \to \IG^k(n+1)$, which is equivariant with respect to the corresponding inclusion $T_n \to T_{n+1}$. Consider its pullback
\[
j_n^*: \CK^*_{T_{n+1}}(\IG^k(n+1)) \to \CK^*_{T_n}(\IG^k(n)),
\]
and then we have
\begin{equation}\label{stabSch}
j_n^*[\Omega_{\lambda}^X]_{T_{n+1}} = 
\begin{cases}
[\Omega_{\lambda}^X]_{T_n} & \mbox{ if } \lambda\in \SP^k(n),\\
0 & \mbox{ if } \lambda\not\in \SP^k(n).
\end{cases}
\end{equation}
Consider the graded projective limit with respect to $j_n^*$
\[
\bbK^{\infty}_T(\IG^k):= \bigoplus_{m\in \ZZ} \lim_{\longleftarrow \atop{n}} \CK^m_{T_n}(\IG^k(n)),
\]
and then by (\ref{stabSch}) one obtains a unique element $[\Omega_{\lambda}^X]_T$ as a limit of the classes $[\Omega_{\lambda}^X]_{T_n}$. Since the Schubert classes form a $\CK^*_{T_n}$-module basis, we can conclude the following.
\begin{lem}\label{lemLimBasis}
Any element $f$ of $\bbK^{\infty}_T(\IG^k)$ can be expressed uniquely as a possibly infinite $\QQ[\beta][[b]]_{\gr}$-linear combination of the classes $[\Omega_{\lambda}^X]_T$:
\[
f = \sum_{\lambda\in \SP^k} c_{\lambda} [\Omega_{\lambda}^X]_T, \ \ \ \ \ c_{\lambda}\in \QQ[\beta][[b]]_{\gr}.
\]
\end{lem}
\begin{defn}\label{dfFiniteCKT}
Let  $\bbK^{\fin}_T(\IG^k)$ be the $\calR_b$-submodule of $\bbK^{\infty}_T(\IG^k)$ consisting of the following possibly infinite $\calR_b$-linear combinations of the classes $[\Omega_{\lambda}^X]_T$
\[
f = \sum_{\lambda\in \SP^k} c_{\lambda} [\Omega_{\lambda}^X]_T, \ \ \ \ \ c_{\lambda}\in \calR_b.
\]
\end{defn}
It will be shown below that there is an isomorphism of $\calR_b$-algebras $\calK_{\infty}^{(k)} \cong \bbK^{\fin}_T(\IG^k)$.
\subsection{Pfaffian formula of the equivariant Schubert classes}\label{4-4}
Let $U$ be the tautological isotropic bundle of $\IG^k(n)$ of rank $n-k$. We abuse notation and denote by $F^i$ the vector bundle over $\IG^k(n)$ with the fiber $F^i$.  We have  
\begin{eqnarray*}
c(E/F^{\ell};u)&=& \prod_{i=1}^n(1+ \bar b_i u)\cdot \prod_{i=1}^{\ell}(1 + b_iu) \ \ \  \ \ \ (0\leq \ell \leq n),\\
c(E/F^{-\ell};u)&=& \prod_{i=\ell+1}^n(1+ \bar b_i u) \ \ \ \ \ \ \ \ (1\leq \ell \leq n), 
\end{eqnarray*}
where $E$ denotes the vector bundle with fiber $E^{(n)}$. Let us point out that for the odd orthogonal case, the action of $T_n$ on $\Span\{\vece_0\}$ is trivial so that $c((F^0)^{\perp}/F^0;u)=1$.
\begin{defn}
We define the classes $\scX_m^{(\ell)} \in \CK^*_{T_n}(\IG^k(n))$ for $m\in \ZZ$ and $\ell=-n,\cdots,n$ as follows. For the symplectic case $\IG^k(n)=\SG^k(n)$, we let 
\[
\scX_m^{(\ell)}:=\scC_m^{(\ell)} = \scS_{m}(U^{\vee}-(E/F^{\ell})^{\vee})
\]
and for the odd orthogonal case $\IG^k(n)=\OG^k(n)$, we let
\[
\scX_m^{(\ell)}:=\scB_m^{(\ell)} = \begin{cases}
\scS_m((U- E/F^{\ell})^{\vee}) & (-n\leq \ell <0),\\
\displaystyle\frac{1}{2} \sum_{s\geq 0} \left(\displaystyle\frac{-\beta}{2}\right)^s\scS_{m+s}((U- E/F^{\ell})^{\vee}) & (0\leq \ell \leq n).
\end{cases}
\]
\end{defn}

Let us point out that by Definition \ref{dfrelsegre} and Remark \ref{remgrlw}, we have
\begin{equation}\label{Cclassgen}
\sum_{m\in \ZZ}\scC_m^{(\ell)} u^m= \frac{1}{1+\beta u^{-1}} \frac{c(E/F^{\ell}; u+\beta)}{c(U;u+\beta)}.
\end{equation}

\begin{thm}[\cite{HIMN}]\label{thmPf}
Let $\lambda$ be a $k$-strict partition in $\SP^k(n)$ of length $r$ and $\chi$ its characteristic index. In $\CK^*_{T_n}(\IG^k(n))$, the Schubert class $[\Omega_{\lambda}^X]_{T_n}$ is given by 
\[
[\Omega_{\lambda}^X]_{T_n}= \sum_{I \subset D(\lambda)} \Pf\left(\sum_{p,q\in \ZZ \atop{p\geq 0, p+q\geq 0}} \ff_{pq}^{ij,I} \scX_{\lambda_i+d_i^I+p}^{(\chi_i)}\scX_{\lambda_j+d_j^I+q}^{(\chi_j)}\right)_{1\leq i<j\leq m},
\]
where  $m=r$ if $r$ is even and $m=r+1$ if $r$ is odd, and $\scX_{-i}^{(-n-1)}:=(-\beta)^i$ for $i\geq 0$. 
\end{thm}
This theorem follows from \cite[Theorem 5.20, 6.16]{HIMN}. Indeed, let $BT_n$ be the classifying space of $T_n$ and $ET_n\rightarrow BT_n$ the universal bundle. Consider the bundle $ET_n\times_{T_n}E$ over $ET_n\times_{T_n}\IG^k(n)$. We can apply Theorem 5.20 or 6.16 in \cite{HIMN} to every finite approximation of this bundle. Then the functoriality of Chern classes implies the claim.

\section{GKM descriptions in algebra and geometry}\label{sec5}
We use techniques of the equivariant localization maps to study both rings $\calK_{\infty}^{(k)}$ and $\bbK^{\infty}_T(\IG^k)$. By using this we have so-called GKM (after Goresky--MacPherson--Kottwitz \cite{GKM}) description for these rings. These maps also enable us to establish the connection between $\GC/\GB$-functions and the stable limits of the torus equivariant Schubert classes.
\subsection{GKM description for $\calK_{\infty}^{(k)}$}\label{sec5-1}
First we study the GKM description of the ring $\calK_{\infty}$.  Let $\Fun(W_{\infty}, \calR_b)$ be the algebra of maps from $W_{\infty}$ to $\calR_b$ whose algebra structure is naturally given by the one of $\calR_b$. An element $\psi$ in $\Fun(W_{\infty}, \calR_b)$ is denoted by  $(v \mapsto \psi(v))_{v \in W_{\infty}}$. We define an action of $W_{\infty}$ on $\Fun(W_{\infty}, \calR_b)$ by
\[
w(\psi)(v):= \psi(vw), \ \ \psi\in \Fun(W_{\infty}, \calR_b), w, v\in W_{\infty}.
\]
\begin{defn}\label{dflocPhi_v}
We introduce the following homomorphism of $\calR_b$-algebras
\[
\Phi_{\infty} : \calK_{\infty} \to \Fun(W_{\infty}, \calR_b); \ \ \ \ \Phi_{\infty}(f) := (v \mapsto \Phi_v(f))_{v \in W_{\infty}},
\]
where $\Phi_v: \calK_{\infty} \to \calR_b$ is the $\calR_b$-algebra homomorphism given by the substitution
\[
x_i \mapsto \begin{cases}  
b_{v(i)} & \mbox{ if $v(i)<0$}\\
0 & \mbox{ if $v(i)>0$}
\end{cases}
 \ \ \ \ \mbox{ and } \ \  a_i \mapsto b_{-v(i)}.
\]
where we set $b_{m}:=\bar b_{-m}$ if $m<0$.
\end{defn}
\begin{defn}\label{dferoot}
Recall that $\Delta^+$ denotes the root system of type B and C in the lattice $L=\bigoplus_{i=1}^\infty \ZZ \varepsilon_i$ (see \S\ref{sec2-1}). We define a map $e: L \to \calR_b$ by
\[
e(\varepsilon_i)=b_i,\ \  e(-\varepsilon_i)=\bar b_i, \ \  e(\alpha+\gamma)=e(\alpha)\oplus e(\gamma) \ \mbox{ and }  \ e(\alpha-\gamma)=e(\alpha)\ominus e(\gamma) \ \ \ \ (\alpha,\gamma\in L).
\]
For each $\alpha \in \Delta^+_n$ one has $e(\alpha)\in \CK^*_{T_n}\cong \QQ[\beta][[b_1,\dots,b_n]]_{\gr}$.  
\end{defn}
\begin{defn} 
Let $\frakK_{\infty}$ be the subalgebra of $\Fun(W_{\infty}, \calR_b)$ consisting of maps $\psi$ such that
\[
\psi(s_{\alpha} v) - \psi(v) \in e(\alpha)\cdot \calR_b, \ \ \mbox{ for all } v\in W_{\infty}\ \ \mbox{ and } \alpha \in \Delta^+.
\]
\end{defn}
\begin{rem}
By the fact that $e(2\varepsilon_i) = b_i \oplus b_i = b_i(2+ \beta b_i)$ and since $2+\beta b_i$ is invertible in $\calR_b$, we can see that $\frakK_n^{(k)}$ is independent of the type B and C. 
\end{rem}
\begin{lem}\label{lemfullinj}
The map $\Phi_{\infty}$ is injective and its image lies in $\frakK_{\infty} \subset \Fun(W_{\infty}, \calR_b)$.
\end{lem}
\begin{proof}
The proof of the latter claim is similar to the one in the proof of Lemma 7.1. in \cite{IkedaNaruse}. We leave the details to the reader. Below we prove the injectivity.

By the definition of $\calK_{\infty}$ and Proposition \ref{propfb}, a homogeneous element $f$ of $\calK_{\infty}$ of degree $d$ can be uniquely written as
\begin{equation}\label{eqf1}
f=\sum_{\lambda\in \SP^0} c_{\lambda}(a;b) \GP_{\lambda}(x), \ \ \ c_{\lambda}(a;b)\in(\calR_a \otimes_{\QQ[\beta]} \calR_b)_{d - |\lambda|}.
\end{equation}
By the definition of $\calR_a \otimes_{\QQ[\beta]} \calR_b$, there exist $m,n$ such that for {\it all} $\lambda \in \SP^0$
\[
c_{\lambda}(a;b)=c_{\lambda}(a_1,\dots,a_n;b_1,\dots, b_m) \in \QQ[\beta][[a_1,\dots,a_n, b_1,\dots,b_m]]_{d-|\lambda|}.
\]
Now suppose that $\Phi_{\infty}(f)=0$. We choose an integer $N > m+n$ and consider an element $v$ of $W_{\infty}$ given by
\[
v=(m+1,\dots,m+n,1,\dots,m, \overline{m+n+1},\dots,\overline{m+n+N}). 
\]
Applying $\Phi_v$ to (\ref{eqf1}) we obtain
\[
\Phi_v(f)=\sum_{\lambda\in \SP^0} c_{\lambda}(\bar b_{m+1},\dots,\bar b_{m+n};b_1,\dots, b_m) \GP_{\lambda}(\bar b_{m+n+1},\dots, \bar b_{m+n+N},0,0,\dots)=0.
\]
Finally we can conclude that $c_{\lambda}(\bar b_{m+1},\dots,\bar b_{m+n};b_1,\dots, b_n)=0$ for all $\lambda \in \SP^0$ by the facts that $\GP_{\lambda}(x_1,\dots, x_N), \lambda\in \SP^0_N$ form a formal basis of $\GGamma_N$  (see the proof of Proposition \ref{propfb}) and that $N$ can be chosen arbitrary as long as it is greater than $m+n$. This completes the proof.
\end{proof}
Next we will obtain the GKM description of $\calK_{\infty}^{(k)}$. 
\begin{defn}
Let $\frakK^{(k)}_{\infty}$ be the subalgebra of $\Fun(\SP^{k}, \calR_b)$ consisting of functions $\psi$ such that
\[
\psi(s_{\alpha} \mu) - \psi(\mu) \in e(\alpha)\cdot \calR_b, \ \ \mbox{ for all } \mu\in \SP^{k}\ \ \mbox{ and } \alpha \in \Delta^+.
\]
\end{defn}
The following proposition is essentially a consequence of Lemma \ref{lemfullinj}.
\begin{prop}\label{universal localization}
The map $\Phi_{\infty}$ naturally induces an injective $\calR_b$-algebra homomorphism 
\[
\Phi_{\infty}^{(k)}: \calK_{\infty}^{(k)} \to \Fun(\SP^k, \calR_b)
\]
and its image lies in $\frakK_{\infty}^{(k)}$.
\end{prop}
\begin{proof}
Since the map $\Phi_{\infty}$ is $W_{\infty}$-equivariant (cf. \cite[Proposition 7.3]{IkedaMihalceaNaruse}), we find that $\Phi_{\infty}$ restricts to the injective map $\calK_{\infty}^{(k)} \to \Fun(W_{\infty}, \calR_b)^{W_{(k)}}$ by taking the $W_{\infty}^{(k)}$-invariant parts. Now we can identify $\Fun(\SP^k, \calR_b)$ with $\Fun(W_{\infty}, \calR_b)^{W_{(k)}}$ as $\calR_b$-algebras as follows. For each $\lambda\in \SP^k$, let $w_{\lambda}$ be the corresponding $k$-Grassmannian element in $W_{\infty}^{(k)}$. For each $\psi \in \Fun(\SP^k, \calR_b)$, the corresponding element of $\Fun(W_{\infty}, \calR_b)^{W_{(k)}}$ is a map sending each $v \in W_{\infty}$  to $\psi(\lambda)$ if $v\in w_{\lambda}W_{(k)}$. Thus we obtain an injective $\calR_b$-algebra homomorphism $\Phi_{\infty}^{(k)}: \calK_{\infty}^{(k)} \to \Fun(\SP^k, \calR_b)$. Furthermore, we can observe that the $W_{\infty}$-action on $\Fun(W_{\infty}, \calR_b)$ preserves $\frakK_{\infty}$ and notice that, under the identification $ \Fun(W_{\infty}, \calR_b)^{W_{(k)}}\cong \Fun(\SP^k, \calR_b)$,  the $W_{(k)}$-invariant part of $\frakK_{\infty}$ coincides with $\frakK_{\infty}^{(k)}$. This completes the proof.
\end{proof}
\subsection{Map $\Psi_{\infty}^{(k)}$ through GKM descriptions}\label{sec5-2}
First we recall the GKM description of $\CK^*_{T_n}(\IG^k(n))$ following \cite[Corollary (3.20)]{KostantKumar} (cf. \cite[Theorem 7.8 ]{Krishna}). 

The set $(\IG^k(n))^{T_n}$ of $T_n$-fixed  points in $\IG^k(n)$ is bijective to $\SP^k(n)$. For each $\lambda \in \SP^k(n)$, let $e_\lambda$ denote the corresponding fixed point. Let $\Fun(\SP^{k}(n),\CK^*_{T_n})$ be the algebra of maps from $\SP^{k}(n)$ to $\CK^*_{T_n}$ whose algebra structure is naturally given by the one of $\CK^*_{T_n}$. We can identify $\CK^*_{T_n}((\IG^k(n))^{T_n})$ with $\Fun(\SP^{k}(n),\CK^*_{T_n})$ as graded $\CK^*_{T_n}$-algebras. The inclusion  $\iota_n: (\IG^k(n))^{T_n} \inc \IG^k(n)$ defines, by pull-back, an injective homomorphism of $\CK^*_{T_n}$-algebras
\[
\iota_n^*: \CK^*_{T_n}(\IG^k(n))\to \Fun(\SP^{k}(n),\CK^*_{T_n}).
\]
Let $\frakK_n^{(k)}$ be the graded $\CK^*_{T_n}$-subalgebra of $\Fun(\SP^k(n), \CK^*_{T_n})$ defined as follows. A map $\psi: \SP^k(n) \to \CK^*_{T_n}$  is in $\frakK_n^{(k)}$ if and only if 
\[
\psi(s_{\alpha}\mu)-\psi(\mu) \in e(\alpha)\cdot \CK^*_{T_n} \ \ \mbox{ for all } \mu\in \SP^k(n) \ \mbox{ and }\  \alpha \in \Delta^+_n
\]
where $e(\alpha)$ was introduced at Definition \ref{dferoot}. Then the image of the map $\iota_n^*$ coincides with $\frakK_n^{(k)}$. In other words, we have the isomorphism of $\CK^*_{T_n}$-algebras
\begin{equation}\label{GKMgeom}
\iota_n^*:\CK^*_{T_n}(\IG^k(n)) \cong \frakK_n^{(k)}. 
\end{equation}

Now we construct an injective homomorphism from $\calK_{\infty}^{(k)}$ to $\bbK^{\infty}_T(\IG^k)$ using the injective homomorphism $\Phi^{(k)}_{\infty}$ given in Proposition \ref{universal localization} together with the isomorphism $\iota_n^*$ at (\ref{GKMgeom}).
\begin{prop}\label{propvarpi}
There is an injective homomorphism of graded $\calR_b$-algebras
\[
\Psi_{\infty}^{(k)}:\calK_{\infty}^{(k)}   \to\bbK^{\infty}_T(\IG^k).
\]
\end{prop}
\begin{proof}
Below all the maps are considered as homomorphisms of graded $\calR_b$-algebras. There is a natural map $\frakK^{(k)}_{\infty} \to \frakK^{(k)}_n$ defined by restricting the domain of each function of $\frakK^{(k)}_{\infty}$ from $\SP^{k}$ to $\SP^{k}(n)$ and projecting its values from $\calR_b$ to $\CK^*_{T_n}$. Similarly, one has maps $\frakK^{(k)}_{n+1}  \to \frakK^{(k)}_n$ for all $n$. By the commutativity of these maps, we obtain an injection
\[
\frakK^{(k)}_{\infty} \to \lim_{\longleftarrow\atop{n}} \frakK^{(k)}_n,
\]
where the limit on the right hand side is the direct sum of the projective limits of each graded piece. On the other hand, the isomorphism $\iota_n^*$ naturally induces an isomorphism 
\[
\bbK^{\infty}_T(\IG^k) \cong \lim_{\longleftarrow\atop{n}} \frakK^{(k)}_n.
\]
Composing the above maps with $\Phi_{\infty}^{(k)}$, we obtain the desired injective homomorphism:
\[
\Psi_{\infty}^{(k)}: \calK_{\infty}^{(k)} \stackrel{\Phi_{\infty}^{(k)}}{\longrightarrow} \frakK^{(k)}_{\infty} \longrightarrow \lim_{\longleftarrow\atop{n}} \frakK^{(k)}_n \cong \bbK^{\infty}_T(\IG^k).
\]
\end{proof}
\subsection{$\GC_{\lambda}$ and $\GB_{\lambda}$ represent Schubert classes}\label{sec5-3}
Recall the following notations.
\[
\begin{array}{c|c|c|c|c|c}
		& \IG^k(n) & \scX_m^{(\ell)}&\Omega_{\lambda}^X & {}_k\GX_m^{(\ell)} & {}_k\GX_{\lambda} \\
\hline
X=C &\SG^k(n)&\scC_m^{(\ell)}&\Omega_{\lambda}^C&{}_k\GC_m^{(\ell)}&{}_k\GC_{\lambda}\\
\hline
X=B &\OG^k(n)& \scB_m^{(\ell)}&\Omega_{\lambda}^B&{}_k\GB_m^{\,(\ell)}&{}_k\GB_{\!\!\lambda}
\end{array}
\]
Define the $\calR_b$-algebra homomorphisms $\Psi_n^{(k)}$ for $n\geq 1$ by compositions:
\[
\Psi_n^{(k)}: \calK_{\infty}^{(k)} \stackrel{\Phi_\infty^{(k)}}{\longrightarrow} \frakK^{(k)}_{\infty} \longrightarrow \frakK^{(k)}_n \cong \CK^*_{T_n}(\IG^k(n)).
\]
Since $\Psi_n^{(k)}=j_n^*\circ \Psi_{n+1}^{(k)}$, they induce a map $\calK_{\infty}^{(k)} \to \bbK^{\infty}_T(\IG^k)$ which coincides with $\Psi_\infty^{(k)}$.
\begin{lem}\label{propGT=Seg} 
For $-n \leq \ell \leq n$ and $m\in \ZZ$, we have $\Psi_n^{(k)}({}_k\GX_m^{(\ell)}(x,a|b)) = \scX_m^{(\ell)}$. Furthermore, we have
\begin{equation}\label{GT to Omega}
\Psi_n^{(k)}(\GX_{\lambda}(x,a|b)) =
\begin{cases}
[\Omega_{\lambda}^X]_{T_n} & \mbox{if $\lambda \in \SP^k(n)$}\\
0 & \mbox{if $\lambda\not\in \SP^k(n)$.}
\end{cases}
\end{equation}
\end{lem}
\begin{proof}
The first claim is a generalization of \cite[Lemma 10.3]{IkedaMihalceaNaruse}. The proof is analogous and it follows from the comparison of the localizations at $\mu \in \SP^k(n)$. Let $w_{\mu}$ be the element of $W_n^{(k)}$ corresponding to $\mu$ and suppose that its one line notation is given as (\ref{oneline for k grass}). Let $\iota_{\mu}^*: \CK^*_{T_n}(\IG^k(n))\to \CK^*_{T_n}$ be the pullback of the inclusion $\iota_\mu: \{e_\mu\} \to\IG^k(n)$ and $\Phi_{w_{\mu}}$ the map introduced at Definition \ref{dflocPhi_v}. We prove for the symplectic case. It suffices to show that 
\begin{equation}\label{eq:loc=loc}
\Phi_{w_\mu}\left(\sum_{m\in \ZZ}  {}_k\GC_m^{(\ell)} u^m\right)=\iota_{\mu}^*\left(\sum_{m\in \ZZ}  \scC_m^{(\ell)} u^m\right).
\end{equation}
We have
\begin{eqnarray*}
&&\Phi_{w_\mu}\left(\sum_{m\in \ZZ} {}_k\GC_m^{(\ell)} u^m\right)\\
&=&\begin{cases}
\dfrac{1}{1+\beta u^{-1}}  \prod_{i=1}^{s} \dfrac{1 + (u+\beta)\bar b_{\zeta_i}}{1 + (u+\beta)b_{\zeta_i}}  \prod_{i=1}^{k} (1 + (u+\beta) \bar b_{v_i}) \prod_{i=1}^{\ell}(1 + (u+\beta) b_i) & (\ell\geq 0),\\
\dfrac{1}{1+\beta u^{-1}}  \prod_{i=1}^{s} \dfrac{1 + (u+\beta)\bar b_{\zeta_i}}{1 + (u+\beta)b_{\zeta_i}}  \prod_{i=1}^{k} (1 + (u+\beta) \bar b_{v_i}) \prod_{i=1}^{|\ell|}\dfrac{1}{1 + (u+\beta)\bar b_i} & (\ell \leq 0).
\end{cases}
\end{eqnarray*}
On the other hand, if  $\bfz_{k+1},\dots, \bfz_{n}$ are the Chern roots of $U$, by (\ref{Cclassgen}) we have 
\begin{eqnarray*}
\sum_{m\in \ZZ} \scC_m^{(\ell)}   u^m 
&=&
\begin{cases}
\dfrac{1}{1+\beta u^{-1}} \dfrac{\prod_{i=1}^n(1+(u+\beta)\bar b_i)}{\prod_{i=k+1}^n (1+(u+\beta) \mathbf{z}_i)} \prod_{i=1}^{\ell}(1+(u+\beta)  b_i)& (\ell\geq 0),\\
\dfrac{1}{1+\beta u^{-1}}\dfrac{\prod_{i=1}^n(1+(u+\beta)\bar b_i)}{\prod_{i=k+1}^n (1+(u+\beta) \mathbf{z}_i)} \prod_{i=1}^{|\ell|}\dfrac{1}{1+(u+\beta) \bar b_i} & (\ell\leq 0).
\end{cases}
\end{eqnarray*} 
Now we use the identity 
\[
\iota_{\mu}^*(c(U;u)) = \prod_{i=1}^s(1+ b_{\zeta_i} u)\cdot  \prod_{i=1}^{n-k-s} (1+ \bar b_{u_i} u)
\]
which follows from (the proof of) Proposition 10.1 \cite{IkedaMihalceaNaruse} and obtain (\ref{eq:loc=loc}). The proof for the case $\IG^k(n)=\OG^k(n)$ is similar.

Finally we show the latter claim. If $\lambda\in \SP^k(n)$, then $[\Omega_{\lambda}^X]_{T_n}$ and $\GX_{\lambda}$ are given by the same Pfaffian formulas except that the entries are given in terms of $\scX_{m_i}^{(\chi_i)}$ or ${}_k\GX_{m_i}^{(\chi_i)}$ respectively. Therefore (\ref{GT to Omega}) for the case when $\lambda\in \SP^k(n)$ follows from the former claim. For the vanishing, it suffices to show the case when $\lambda\in \SP^k(n+1)\backslash\SP^k(n)$. In that case, by $\Psi_n^{(k)}=j_n^*\circ \Psi_{n+1}^{(k)}$, we have $\Psi_{n}^{(k)}(\GX_{\lambda}(x,a|b))  = j_n^*[\Omega_{\lambda}^X]_{T_{n+1}}$ which is $0$ by (\ref{stabSch}). 
\end{proof}
\begin{thm}\label{thm:iso}
The map $\Psi_{\infty}^{(k)}$ restricts to an $\calR_b$-algebra isomorphism
\[
\calK_{\infty}^{(k)} \cong \bbK^{\fin}_T(\IG^k)
\]
which sends ${}_k\GX_{\lambda}$ to $[\Omega_{\lambda}^X]_T$. In particular, $\bbK^{\fin}_T(\IG^k)$ is an $\calR_b$-algebra.
\end{thm}
\begin{proof}
The injectivity of $\Psi_{\infty}^{(k)}$ and Lemma \ref{lemLimBasis} imply that $\GX_{\lambda}(x,a|b) (\lambda \in \SP^k)$ form a formal basis $\calK_{\infty}^{(k)}$ over $\calR_b$. That is, any element $f$ of $\calK_{\infty}^{(k)}$ can be expressed uniquely as a possibly infinite $\calR_b$-linear combination
\[
f = \sum_{\lambda \in \SP^k} c_{\lambda}^X(b) \GX_{\lambda}(x,a|b), \ \ \ \ c_{\lambda}^X(b) \in \calR_b.
\]
On the other hand, Lemma \ref{propGT=Seg} shows that $\Psi_{\infty}^{(k)}$ sends $\GX_{\lambda}(x,a|b)$ to $[\Omega_{\lambda}^X]_T$. Therefore we can conclude that the image of $\Psi_{\infty}^{(k)}$ coincides with $\bbK^{\fin}_T(\IG^k)$.
\end{proof}
\begin{rem}
If $k=0$ and $\lambda$ is a strict partition, then one result shows that $\GQ_{\lambda}(x|b)$ and $\GP_{\lambda}(x|0,b)$ defined by Ikeda--Naruse in \cite{IkedaNaruse} coincide with the functions ${}_0\GC_{\lambda}(x|b)$ and ${}_0\GB_{\lambda}(x|b)$ respectively, simultaneously establishing their Pfaffian formula.
\end{rem}
 
\vspace{5mm}

\noindent\textbf{Acknowledgements.} 
A considerable part of this work developed while the first and third authors were affiliated to KAIST, which they would like to thank for the excellent working conditions.
A part of this work was developed while the first author was affiliated to  POSTECH, which he would like to thank for the excellent working conditions. He would also like to gratefully acknowledge the support of the National Research Foundation of Korea (NRF) through the grants funded by the Korea government (MSIP) (2014-001824 and 2011-0030044).
The second author  is supported by Grant-in-Aid for Scientific Research (C) 18K03261, 15K04832.
The third author is supported by 	Grant-in-Aid for Young Scientists (B) 16K17584.
The fourth author is supported by Grant-in-Aid for Scientific Research (C) 25400041, (B) 16H03921.

\bibliographystyle{acm}
\bibliography{references}   

\def\cprime{$'$}
\begin{thebibliography}{10}

\bibitem{Anderson2016}
{\sc {Anderson}, D.}
\newblock {K-theoretic Chern class formulas for vexillary degeneracy loci}.
\newblock {\em ArXiv e-prints\/} (Jan. 2017).

\bibitem{BuchKreschTamvakis1}
{\sc Buch, A.~S., Kresch, A., and Tamvakis, H.}
\newblock A {G}iambelli formula for isotropic {G}rassmannians.
\newblock {\em Selecta Math. (N.S.) 23}, 2 (2017), 869--914.

\bibitem{DaiLevine}
{\sc Dai, S., and Levine, M.}
\newblock Connective algebraic {$K$}-theory.
\newblock {\em J. K-Theory 13}, 1 (2014), 9--56.

\bibitem{GKM}
{\sc Goresky, M., Kottwitz, R., and MacPherson, R.}
\newblock Equivariant cohomology, {K}oszul duality, and the localization
  theorem.
\newblock {\em Invent. Math. 131}, 1 (1998), 25--83.

\bibitem{HIMN}
{\sc Hudson, T., Ikeda, T., Matsumura, T., and Naruse, H.}
\newblock Degeneracy loci classes in {$K$}-theory --- determinantal and
  {P}faffian formula.
\newblock {\em Adv. Math. 320\/} (2017), 115--156.

\bibitem{IkedaMatsumura}
{\sc Ikeda, T., and Matsumura, T.}
\newblock Pfaffian sum formula for the symplectic {G}rassmannians.
\newblock {\em Math. Z. 280}, 1-2 (2015), 269--306.

\bibitem{IkedaMihalceaNaruse}
{\sc Ikeda, T., Mihalcea, L.~C., and Naruse, H.}
\newblock Double {S}chubert polynomials for the classical groups.
\newblock {\em Adv. Math. 226}, 1 (2011), 840--886.

\bibitem{IkedaNaruse}
{\sc Ikeda, T., and Naruse, H.}
\newblock {K}-theoretic analogues of factorial {S}chur {P}- and {Q}-functions.
\newblock {\em Adv. Math. 226}, 1 (2011), 840--886.

\bibitem{KrNr}
{\sc Kirillov, A.~N., and Naruse, H.}
\newblock Construction of double {G}rothendieck polynomials of classical types
  using id{C}oxeter algebras.
\newblock {\em Tokyo J. Math. 39}, 3 (2017), 695--728.

\bibitem{KostantKumar}
{\sc Kostant, B., and Kumar, S.}
\newblock {$T$}-equivariant {$K$}-theory of generalized flag varieties.
\newblock {\em J. Differential Geom. 32}, 2 (1990), 549--603.

\bibitem{Krishna}
{\sc Krishna, A.}
\newblock Equivariant cobordism for torus actions.
\newblock {\em Adv. Math. 231}, 5 (2012), 2858--2891.

\bibitem{Krishna0}
{\sc Krishna, A.}
\newblock Equivariant cobordism of schemes.
\newblock {\em Doc. Math. 17\/} (2012), 95--134.

\bibitem{MacdonaldHall}
{\sc Macdonald, I.~G.}
\newblock {\em Symmetric functions and {H}all polynomials}, second~ed.
\newblock Oxford Mathematical Monographs. The Clarendon Press, Oxford
  University Press, New York, 1995.
\newblock With contributions by A. Zelevinsky, Oxford Science Publications.

\bibitem{Tamvakis2011Crelle}
{\sc Tamvakis, H.}
\newblock Giambelli, {P}ieri, and tableau formulas via raising operators.
\newblock {\em J. Reine Angew. Math. 652\/} (2011), 207--244.

\bibitem{TamvakisWilson}
{\sc Tamvakis, H., and Wilson, E.}
\newblock Double theta polynomials and equivariant {G}iambelli formulas.
\newblock {\em Math. Proc. Cambridge Philos. Soc. 160}, 2 (2016), 353--377.

\bibitem{WilsonThesis}
{\sc Wilson, V.}
\newblock {E}quivariant {G}iambelli {F}ormulae for {G}rassmannians.
\newblock Ph.D. thesis. University of Maryland (2010).

\end{thebibliography}

\


\begin{small}
{\scshape
\noindent Thomas Hudson, Fachgruppe Mathematik und Informatik, Bergische Universit\"{a}t Wuppertal, 
42119 Wuppertal, Germany
}
\end{small}

{\textit{email address}: \tt{hudson@math.uni-wuppertal.de}}

\

\begin{small}
{\scshape
\noindent  Takeshi Ikeda, Department of Applied Mathematics, Okayama University of Science, Okayama 700-0005, Japan
}
\end{small}

{\textit{email address}: \tt{ike@xmath.ous.ac.jp}}

\

\begin{small}
{\scshape
\noindent Tomoo Matsumura, Department of Applied Mathematics, Okayama University of Science, Okayama 700-0005, Japan
}
\end{small}

{\textit{email address}: \tt{matsumur@xmath.ous.ac.jp}}

\

\begin{small}
{\scshape
\noindent Hiroshi Naruse, Graduate School of Education, University of Yamanashi, Yamanashi 400-8510, Japan
}
\end{small}

{\textit{email address}: \tt{hnaruse@yamanashi.ac.jp}}

\end{document}